\documentclass{article}
\usepackage[utf8]{inputenc}

\usepackage{makeidx}
 \usepackage{multirow}
\usepackage{amssymb}
\usepackage{amsfonts}
\usepackage{amsmath}
\usepackage{xcolor}
\usepackage{graphicx}
\newtheorem{theorem}{Theorem}

\newtheorem{corollary}[theorem]{Corollary}

\newtheorem{lemma}{Lemma}

\newenvironment{proof}[1][Proof]{\noindent\textbf{#1.} }{\ \rule{0.5em}{0.5em}}

\newcommand{\be}{\begin{equation}}
\newcommand{\ee}{\end{equation}}
\newcommand{\beq}{\begin{eqnarray}}
\newcommand{\eeq}{\end{eqnarray}}
\newcommand{\nbeq}{\begin{eqnarray*}}
\newcommand{\neeq}{\end{eqnarray*}}
\newcommand{\D}{\displaystyle}

\newcommand{\alphab}{{\pmb \alpha}}

\begin{document}

\begin{center}
{\Large\bf Exponential and Hypoexponential Distributions: Some Characterizations}
\end{center}
\begin{center}
{\sc George Yanev}\\
{\it University of Texas Rio Grande Valley\\
Edinburg, Texas, U.S.A.\\
and \\
 Institute of Mathematics and Informatics\\
Sofia, BULGARIA \\
}
e-mail: {\tt george.yanev@utrgv.edu}
\end{center}

\abstract{The (general) hypoexponential distribution is the distribution of a sum of independent exponential random variables. We consider the particular case when the involved exponential variables have distinct rate parameters. We prove that 
the following converse result is true. If~for some $n\ge 2$, $X_1, X_2,\,\ldots,\,X_n$ are independent copies of a random variable $X$ with unknown distribution $F$ and a specific linear combination of $X_j$'s has hypoexponential distribution, then $F$ is exponential. Thus, we obtain new characterizations of the exponential distribution. As corollaries of the main results, we extend some previous characterizations established recently by Arnold and Villase\~{n}or (2013) for a particular convolution of two random variables. 
}

\section{Introduction and Main~Results} 
Sums of exponentially distributed random variables play a central role in many stochastic models of real-world phenomena.  {{\it Hypoexponential distribution} is the convolution of $k$ exponential distributions each with their own rate $\lambda _{i}$, the~rate of the $i^{th}$ exponential distribution. As~an example, consider the distribution of the time to absorption of a finite state Markov process. If~we have a $k+1$ state process, where the first $k$ states are transient and the state $k+1$ is an absorbing state, then the time from the start of the process until the absorbing state is reached is {\it phase-type distributed}. This becomes the hypoexponential if we start in state $1$ and move skip-free from state $i$ to $i+1$ with rate $ \lambda _{i}$ until state $k$ transitions with rate $ \lambda _{k}$ to the absorbing state $k+1$.}

We write $Z_i\sim {\rm Exp}(\lambda_i)$ for $\lambda_i>0$, if~$Z_i$ has density
\[
f_i(z)=\lambda_i {\rm e}^{-\lambda_i z}, \quad z\ge 0 \quad \mbox{\it (exponential distribution)}.
\]

The distribution of the sum $ S_n:=Z_1+Z_2+\ldots +Z_n$,
where $\lambda_i$ for $i=1,\,\ldots,\,n$ are not all identical, is called (general) {\it hypoexponential distribution} (see~\cite{LL19,SKK16}). It is absolutely continuous and we denote by $g_n$ its density. {{It is called the {\it hypoexponetial} distribution as it has a coefficient of variation less than one, compared to the {\it hyper-exponential} distribution which has coefficient of variation greater than one and the {\it exponential} distribution which has coefficient of variation of one.}
In this paper, we deal with a particular case of the hypoexponential distribution when all $\lambda_i$ are distinct, i.e.,~$\lambda _i\ne \lambda _j$ when $i\ne j$. In~this case, it is known (\cite{R19}, p. 311;~\cite{F71}, Chapter 1, Problem 12)) that 
\be \label{ghexp}
S_n=
Z_1+Z_2+\ldots +Z_n
\quad \mbox{has density} \quad 
g_n(z):=\sum_{j=1}^n \ \ell_j f_j(z),
\quad z\ge 0.
\ee
Here the weight $\ell_j$ is defined as
\[
\ell_j=\prod_{i=1, i\ne j}^n \frac{\lambda_i}{\lambda_i-\lambda_j}.
\]

Please note that $\ell_j:=\ell_j(0)$, where
$\ell _{1}(x),\,\dots,\,\ell _{n}(x)$ are identified (see~\cite{SB99}) as the Lagrange basis polynomials associated with the points 
$\lambda_{1},\,\dots,\,\lambda_{n}$.
The convolution density $g_n$ in (\ref{ghexp}) is the weighted average of the values of the densities of $Z_1, Z_2,\,\ldots,\,Z_n$, where the weights $\ell_j$ sum to 1 (see~\cite{SB99}). Notice,~however,~since the weights can be both positive or negative, $g_n$ is not a ``usual'' mixture of densities. If~we place $\lambda_j$'s in  increasing or decreasing order, then the corresponding coefficients $\ell_j$'s alternate in~sign. 

Consider the Laplace transforms $\varphi_i(t):={\mathsf E}[{\rm e}^{-tZ_i}]$, $t\ge 0$, $i=1,2,\,\ldots,\,n$. They are well-defined and will play a key role in the proofs of the main~results.  

To begin with, let us look at the case when all $Z_i$'s are identically distributed, i.e.,~$\lambda_i=\lambda$ for \mbox{$i=1,2,\,\ldots,\,n$}, so we can use $\varphi$ for the common Laplace transform. The~sum \mbox{$S_n=Z_1+Z_2+\ldots +Z_n$} has Erlang distribution whose Laplace transform $\tilde{\varphi}$, because~of the independence, is expressed as~follows:
\[
\tilde{\varphi}(t)=\mathsf{E}\left[ {\rm e}^{-tS_n}\right]=\varphi^n(t) = \left(\frac{\lambda}{\lambda+t}\right)^n.
     \]
     
If we go in the opposite direction, assuming that $S_n$ has Erlang distribution with Laplace transform $\tilde{\varphi}$, then we conclude that 
 $\varphi_i(t)= \lambda(\lambda +t)^{-1}$ for each $i=1,2,\,\ldots,\,n$, which in turn implies that $Z_i\sim {\rm Exp}(\lambda)$. By~words, if~$Z_i$ are independent and identically distributed random variables and their sum has Erlang distribution, then the common distribution is~exponential.
 
{\it Does a similar characterization hold when the rate parameters $\lambda_i$ are all different?} The answer to this question is not obvious. It is our goal in this paper to show that the answer is~positive.

 Let $\mu_1, \mu_2,\,\ldots,\,\mu_n$ be positive real numbers, such that $\lambda_i=\lambda/\mu_i$. Without~loss of generality suppose that $\mu_1>\mu_2>\ldots >\mu_n>0$. Assume that $X_1, X_2,\,\ldots,\,X_n$, for~ fixed $n\ge 2$, are independent and identically distributed as a random variable $X$ with density $f$, $f(x)=\lambda{\rm e}^{-\lambda x}$, $x>0$. 
Then (\ref{ghexp}) is equivalent to the following:
\be \label{eqn_4}
S_n:=\mu_1X_1+\mu_2X_2+\cdots+\mu_nX_{n} \quad \mbox{has density}\quad g_n(x)= \sum_{j=1}^n \frac{\ell_j}{\mu_j}f\left(\frac{x}{\mu_j}\right),\quad x\ge 0.
\ee
Here the coefficients/weights are given as follows: 
\be \label{Lagrange_mu}
\ell_j =\prod_{i=1, i\ne j}^n \frac{\mu^{-1}_i}{\mu^{-1}_i-\mu^{-1}_j}=  \prod_{i=1, i\ne j}^n \frac{\mu_j}{\mu_j-\mu_i}, \quad j=1,2,\,\ldots,\,n.
\ee

We use now the common Laplace transform $\varphi(t):=E[{\rm e}^{-tX_i}]$. Please note that since $\mu_i\ne \mu_j$ for $i\ne j$, relation (\ref{eqn_4}) implies that
\beq \label{wphi}
\hspace{-0.5cm}\varphi(\mu_1 t)\varphi(\mu_2 t)\cdots \varphi (\mu_n t) & = &
\int_0^\infty {\rm e}^{-tx}g_n(x)\, dx\\
    & = & \int_0^\infty {\rm e}^{-tx}\sum_{j=1}^n \frac{\ell_j}{\mu_j}f\left(\frac{x}{\mu_j}\right)\, dx  \nonumber \\
    & = & \sum_{j=1}^n \ell_j\int_0^\infty {\rm e}^{-tx}\frac{1}{\mu_j}f\left(\frac{x}{\mu_j}\right)\, dx  =  \sum_{j=1}^n \ell_j\varphi(\mu_j t). \nonumber
    \eeq
    
The idea now is to start with an arbitrary non-negative random variable $X$ with unknown density $f$ and Laplace transform $\varphi$. If~the Laplace transform of the linear combination $S_n=\sum_{i=1}^n\mu_iX_i$
satisfies (\ref{wphi}), we will derive that  $\varphi (t)= \lambda(\lambda + t)^{-1}$. Thus, the common distribution of $X_j$, $j=1,2,\,\ldots,\,n$ is exponential. More precisely, the~following characterization result~holds.

\begin{theorem} Suppose that $ X_1, X_2,\,\ldots,\,X_n$, $n\ge 2$, are independent copies of a non-negative random variable $X$ with density $f$. Assume further that $X$ satisfies Cram\'{e}r's condition: there is a number $t_0>0$ such that ${\mathsf E}[{\rm e}^{-tX}]<\infty$ for all $t\in (-t_0,t_0)$. If~relation (\ref{eqn_4}) is satisfied for fixed $n\ge 2$ {and fixed positive mutually different numbers $\mu_1, \mu_2,\,\ldots,\,\mu_n$},
then $X\sim {\rm Exp}(\lambda)$ for some $\lambda>0$.
\end{theorem}
The studies of characterization properties of {exponential distributions} are abundant. Comprehensive surveys can be found in~\cite{A17,AH95,AV86,N06}. More recently, Arnold and Villase\~{n}or~\cite{AV13} obtained a series of exponential characterizations 
involving sums of two random variables and conjectured possible extensions for sums of more than two variables (see also~\cite{Y20}). Corollary~1 below extends the characterizations in~\cite{AV13,Y20} to sums of $n$ variables, for~any fixed $n\ge 2$.

Consider the special case of (\ref{eqn_4}) when $\mu_j=1/j$ for $j=1,2,\,\ldots,\,n$. Under~this choice of $\mu_j$'s, the~formula for the weight $\ell_j$ simplifies to (see~\cite{F71}, Chapter 1, Problem 13)
\[
\ell_j  =  \prod_{i=1,  i\ne j}^n \frac{i}{i-j}
   = {n \choose j}(-1)^{j-1}.
\]
Therefore, Theorem~1 reduces to the following~corollary.

\begin{corollary}
Suppose that $X_1, X_2,\,\ldots,\,X_n$, $n\ge 2$, are independent copies of a non-negative random variable $X$ with density $f$. Assume further that $X$ satisfies Cram\'{e}r's condition: there is a number $t_0>0$ such that ${\mathsf E}[{\rm e}^{-tX}]<\infty$ for all $t\in (-t_0,t_0)$. If~for fixed $n\ge 2$, 
\be \label{main_thm}
X_1+\frac{1}{2}X_2+\ldots+\frac{1}{n}X_n \quad \mbox{has density}\quad  \sum_{j=1}^n {n \choose j} (-1)^{j-1}jf(jx), \quad x\ge 0,
\ee
then $X\sim {\rm Exp}(\lambda)$ for some $\lambda>0$.
\end{corollary}

 The exponential distribution has the striking property that if $\lambda=1$ {\it ({unit} exponential)}, then the density $f$ equals the survival function {(the tail of the cumulative distribution function)} $\overline{F}=1-F$. Therefore, in~case of {unit} exponential distribution, (\ref{eqn_4}) can be written as follows:
\be \label{eqn_44}
\tilde{S}_n:=\mu_1 X_1+\mu_2 X_2+\cdots+\mu_n X_{n}  \mbox{has density} \tilde{g}_n(x):= \sum_{j=1}^n \frac{\ell_j}{\mu_j}
\overline{F}\left(\frac{x}{\mu_j}\right), x\ge 0.
\ee 

We will show that (\ref{eqn_44}) is a sufficient condition for $X_1, X_2,\,\ldots,\,X_n$ to be {\it {unit} exponential}. 

\begin{theorem} Suppose that $X_1, X_2,\,\ldots,\,X_n$, $n\ge 2$, are independent copies of a non-negative random variable $X$ with distribution function $F$.  Assume also that $X$ satisfies Cram\'{e}r's condition: there is a number $t_0>0$ such that ${\mathsf E}[{\rm e}^{-tX}]<\infty$ for all $t\in (-t_0,t_0)$. If~relation (\ref{eqn_44}) is satisfied for fixed $n\ge 2$, 
then $X\sim {\rm Exp}(1)$.
\end{theorem}
Setting $\mu_j=1/j$ for $j=1,2,\,\ldots,\,n$, we obtain the following corollary of Theorem~2.

\begin{corollary}
 Suppose that $X_1, X_2,\,\ldots,\,X_n$, $n\ge 2$, are independent copies of a non-negative random variable $X$ with distribution function $F$.  Assume also that $X$ satisfies Cram\'{e}r's condition: there is a number $t_0>0$ such that ${\mathsf E}[{\rm e}^{-tX}]<\infty$ for all $t\in (-t_0,t_0)$.  If~for fixed $n\ge 2$, 
\be \label{main_thmc}
X_1+\frac{1}{2}X_2+\ldots+\frac{1}{n}X_n \quad \mbox{has density}\quad \sum_{j=1}^n {n \choose j} (-1)^{j-1}j\overline{F}(jx) \qquad x>0,
\ee
then $X\sim {\rm Exp}(1)$ for some $\lambda>0$.
\end{corollary}

We organize the rest of the paper as follows. Section~\ref{sec:2} contains preliminaries needed in the proofs of the theorems. The~proofs themselves are given in Section~\ref{sec:3}. We discuss the findings in the concluding Section~\ref{sec:4}.

\section{Auxiliaries}
\label{sec:2}

    We will need the Leibniz rule for differentiating {a} product of functions.
Denote by $v^{(k)}$ the $k$th derivative of $v(x)$ with $v^{(0)}(x):=v(x)$. 
Let us define a multi-index set
$\alphab=(\alpha_1,\alpha_2,\,\ldots,\,\alpha_{n}) $ as an $n$-tuple of non-negative integers, and~denote
$| \alphab |=\alpha_1+\alpha_2+\ldots+\alpha_n$. 
Leibniz considered the problem of determining the $k$th derivative of the product of $n$ smooth functions $v_1(t)v_2(t)\cdots v_n(t)$ and obtained the formula (e.g. \cite{TL03})
\be \label{Lm}
\frac{{\rm d}^k}{{\rm d}t^k}\left(\prod_{i=1}^n v_i(t)\right) = \sum_{|\alphab|=k} \left(\frac{k!}{\alpha_1!\alpha_2!\cdots \alpha_n!} \prod_{i=1}^n v_i^{(\alpha_i)}(t)\right). 
\ee
Here the summation is taken over all multi-index sets $\alphab$ with $|\alphab|=k$. Formula (\ref{Lm}) can easily be proved by induction.

\begin{lemma} Assume that $v(t)=\sum_{i=0}^\infty a_it^i$ is a functional series, such that for some $\tilde{t}_0>0$, the~$k^{th}$ order derivative $v^{(k)}(t)$ exists for all $t\in (-\tilde{t}_0, \tilde{t}_0)$. Then for arbitrary positive real constants $\mu_1, \mu_2,\,\ldots,\,\mu_n$, we have
\be \label{newf}
\frac{{\rm d}^k}{{\rm d}t^k}\left(\prod_{i=1}^n v(\mu_it)\right)
\Big|_{t=0}= k!\sum_{|\alphab|=k} \prod_{i=1}^n \mu_i^{\alpha_i}a_{\alpha_i}.
\ee
\end{lemma}
\begin{proof} Formula (\ref{newf}) is proved by applying Leibniz rule (\ref{Lm}) to $\prod_{i=1}^n v(\mu_it)$.
\end{proof}

In addition to (\ref{newf}), we will need some properties of Lagrange basis polynomials  $\ell_j$ collected~below.

\begin{lemma} (see~\cite{SKK13})
 Let $\lambda_1, \lambda_2,\,\ldots,\,\lambda_n$ be positive real numbers, such that 
 $\lambda_i\ne \lambda_j$ for $i\ne j$. Denote
 \[
\ell_j =  \prod_{i=1, i\ne j}^n \frac{\lambda_i}{\lambda_i-\lambda_j} \qquad j=1,2,\,\ldots,\,n.
\]
 Then, for~$n\ge 2$, we have the following:
 \begin{description}
 \item {\rm (i)} $\quad \D \sum_{j=1}^n \ell_j=1$.
 \item {\rm (ii)}   $\quad \D \sum_{j=1}^n \ell_j\lambda_j^k=0 \quad \mbox{for any}\ k, \ \  1\le k \le n-1$.
 \item {\rm (iii)} $\quad \D \sum_{j=1}^n  \frac{\ell_j}{\lambda_j^k}\ge \sum_{j=1}^n \frac{1}{\lambda_j^k} \quad \mbox{for any}\ k, \ \  1\le k \le n-1$,
where the equality holds if and only if $k=1$.
\end{description}
\end{lemma}
\begin{proof} Claim (i) {follows by integrating (\ref{ghexp}) over $z>0$}. Claim (ii) is proved in Corollary~1 of~\cite{SKK13}.  To~prove claim (iii) we involve  $\alphab$, the~multi-index set as in (\ref{Lm}). 
For $k\ge 1$, we have
$\alphab=\alphab' \cup \alphab''$, where
\nbeq
\alphab' & = & \{|\alphab|=k : \ \mbox{only one index in $\alphab$ equals $k$ and all others are zeros}\} \\
\alphab'' & = & \{|\alphab|=k : \ \mbox{no single index in $\alphab$ equals $k$}\}.
\neeq

According to Proposition~5 in~\cite{SKK13} we obtain, for~$n\ge 2$ and $k\ge 1$, the~following chain of relations:
\beq \label{skk}
\sum_{j=1}^n  \frac{\ell_j}{\lambda_j^k} & = & \sum_{|\alphab|=k} \frac{1}{\lambda_1^{\alpha_1}\lambda_2^{\alpha_2}\cdots \lambda_n^{\alpha_n}} \\
& = & \sum_{|\alphab'|} \frac{1}{\lambda_1^{\alpha_1}\lambda_2^{\alpha_2}\cdots \lambda_n^{\alpha_n}} +\sum_{|\alphab''|} \frac{1}{\lambda_1^{\alpha_1}\lambda_2^{\alpha_2}\cdots \lambda_n^{\alpha_n}} \nonumber \\
& = & \sum_{j=1}^n  \frac{1}{\lambda_j^k}+\sum_{|\alphab''|} \frac{1}{\lambda_1^{\alpha_1}\lambda_2^{\alpha_2}\cdots \lambda_n^{\alpha_n}} \nonumber \\
& \ge  & 
\sum_{j=1}^n  \frac{1}{\lambda_j^k}. \nonumber
\eeq
Clearly, the~equality in (\ref{skk}) holds if and only if $k=1$. The~proof is complete.
\end{proof}

The properties in Lemma~2 can be easily verified, as~an illustration, for~$n=2$, $k=1$, and~\mbox{$k=2$}.~Indeed, 
\nbeq
\sum_{j=1}^2 \ell_j & = & \frac{\lambda_2}{\lambda_2-\lambda_1}+\frac{\lambda_1}{\lambda_1-\lambda_2}=1, \qquad
\sum_{j=1}^2 \ell_j\lambda_j = \frac{\lambda_2\lambda_1}{\lambda_2-\lambda_1}+\frac{\lambda_1\lambda_2}{\lambda_1-\lambda_2}=0, \\
\sum_{j=1}^2 \frac{\ell_j}{\lambda_j}  & =  & \frac{\lambda_2}{(\lambda_2-\lambda_1)\lambda_1}+\frac{\lambda_1}{(\lambda_1-\lambda_2)\lambda_2}
     =  \frac{\lambda_2 + \lambda_1}{\lambda_1\lambda_2} 
     =  \sum_{i=1}^2\frac{1}{\lambda_i}, \\
\sum_{j=1}^2 \frac{\ell_j}{\lambda_j^2}  & = & \frac{\lambda_2}{(\lambda_2-\lambda_1)\lambda_1^2}+\frac{\lambda_1}{(\lambda_1-\lambda_2)\lambda_2^2}
     =  \frac{\lambda_2^2+\lambda_2\lambda_1+\lambda_1^2}{\lambda_1^2\lambda_2^2} 
     =  \sum_{i=1}^2\frac{1}{\lambda_i^2} + \frac{1}{\lambda_1\lambda_2}.
\neeq

\section{Proofs of the Characterization~Theorems} 
\label{sec:3}

In the proofs of both theorems we follow the four-step~scheme.
\begin{itemize}
\item Consider $X_1, X_2,\,\ldots,\,X_n$ for $n\ge 2$ to be independent copies of a non-negative random variable $X$ with density $f$. Suppose $\mu_1>\mu_2>\ldots >\mu_n$ are positive real numbers.
\item Assume the characterization property
\[
S_n=\mu_1X_1+\mu_2X_2+\cdots+\mu_nX_{n} \quad \mbox{has density}\quad g_n(x)=\sum_{j=1}^n \frac{\ell_j}{\mu_j}f\left(\frac{x}{\mu_j}\right),\quad x\ge 0,
\]
where $\ell_j$ is given in (\ref{Lagrange_mu}).
\item For the Laplace transform $\varphi(t)={\mathsf E}[{\rm e}^{-tX}]$, $t\ge 0$, obtain the equation
\be \label{eqn1}
\varphi(\mu_1 t)\varphi(\mu_2 t)\cdots \varphi (\mu_n t) = \sum_{j=1}^n \ell_j\varphi(\mu_j t).
\ee
\item Using Leibniz rule for differentiating product of functions and properties of Lagrange basis polynomials, show that  (\ref{eqn1}) has a unique solution given by
$\varphi(t)=(1+\lambda^{-1} t)^{-1}$ for some $\lambda>0$ and conclude that
\[
X_1, X_2,\,\ldots,\,X_n \quad \mbox{are}\quad {\rm Exp}(\lambda)\quad \mbox{random variables.}
\]
\end{itemize}

\begin{proof}[Proof of Theorem 1]
Recall that (see (\ref{wphi})) 
\[
\varphi(\mu_1 t)\varphi (\mu_2 t)\cdots \varphi (\mu_n t) 
      = \sum_{j=1}^n \ell_j\varphi (\mu_j t). 
\]

Dividing  both sides of this equation by $\varphi(\mu_1t)\varphi(\mu_2t)\cdots \varphi(\mu_nt)$, we obtain
\be\label{eqn24}
1 = \sum_{j=1}^n \left(\ell_j\prod_{i=1, i\ne j}^n \psi (\mu_i t)\right), 
\ee
where $\psi:=1/\varphi$. Consider the series
\be \label{notation4}
\psi(t)=\sum_{k=0}^\infty a_kt^k,
\ee
which, as~a consequence of Cram\'{e}r's condition for $\varphi$, is convergent in a proper 
neighborhood of $t=0$.
To 
prove the theorem, it is sufficient to show that
\be \label{claim}
\psi(t)=1+\lambda^{-1} t, \qquad \lambda>0.
\ee

We will prove that (\ref{eqn24}) implies (\ref{claim}) by showing that the
coefficients $\{a_k\}_{k=0}^\infty$ in (\ref{notation4}) satisfy 
$a_0=1$, $a_1=\lambda^{-1}>0$, and~$a_k=0$ for $k\ge 2$.
Notice first that
\be \label{a0}
a_0=\frac{1}{\varphi(0)}=1.
\ee
Denote 
\[
\Psi_j(t):= \prod_{i=1, i\ne j}^n  \psi(\mu_it) \quad \mbox{and} \quad 
H(t)  :=  \sum_{j=1}^n \ell_j \Psi_j(t) = \sum_{k=0}^\infty h_kt^k.
\]
By (\ref{eqn24})  we have $H(t)\equiv 1$ and therefore $h_0=1$ and $h_k=0$ for all $k\ge 1$. Equating   $h_k$'s  to the corresponding coefficients of the series in the right-hand side of (\ref{eqn24}), we will obtain equations for $\{a_k\}_{k=0}^\infty$. As~a first step, note that 
\be \label{coeff14}
h_k =  \frac{1}{k!}H^{(k)}(t)|_{t=0} 
     =  
 \frac{1}{k!}\sum_{j=1}^n \ell_j\Psi^{(k)}_j(t)\big|_{t=0},\qquad k\ge 1. 
 \ee
 
 Next, we apply Leibniz rule for differentiation.
To fix the notation, let us define a multi-index set
$\alphab_{-j}=(\alpha_1,\,\ldots,\,\alpha_{j-1},\alpha_{j+1},\,\ldots,\,\alpha_{n})$, $1\le j\le n$ as a set of $(n-1)$-tuples of non-negative integer numbers, with~$| \alphab_{-j} |=\alpha_1+\ldots+\alpha_{j-1}+\alpha_{j+1}+\ldots+\alpha_n$.
Applying Lemma~1 for fixed $k\ge 1$ and fixed $1\le j\le n$, we obtain
\be \label{hk}
 \Psi^{(k)}_j(t)\big|_{t=0} 
    = 
   k! \sum_{\{|\alphab_{-j}|=k\}} \prod_{i=1, i\ne j}^n \mu_i^{\alpha_i}a_{\alpha_i}. 
\ee

Introduce the set $\Lambda_{k,j}:=\{\alphab_{-j} : |\alphab_{-j} |=k\}$ and partition it into three disjoint subsets as follows:
\[
\Lambda_{k,j}=\Lambda'_{k,j}\cup \Lambda''_{k,j}\cup \Lambda'''_{k,j},
\]
where {for $k\ge 1$}
\nbeq
\Lambda'_{k,j}  & =  &\{|\alphab_{-j}|=k : \ \mbox{only one index in $\alphab_{-j}$ equals $k$, all others are zeros} \}\\
\Lambda''_{k,j} &  =  & \{|\alphab_{-j}|=k : \ {k\ge 2}\  \mbox{and exactly $k$ of the indices in $\alphab_{-j}$ equal $1$, all others are zeros}\} \\
\Lambda'''_{k,j} & =  &\{|\alphab_{-j}|=k : \ {k\ge 3}\ \mbox{{ and there is an index $\alpha_i$ with $2\le \alpha_i<k$}} \}.
\neeq

{For example, if~$n=5$, $k=3$, and~$j=5$, then 
$\Lambda'_{3,5}   =  \{ (3, 0, 0, 0), (0, 3, 0, 0), (0, 0, 3, 0), (0, 0, 0, 3)\}$, 
$\Lambda''_{3,5}   =   \{ (1, 1, 1, 0), (1, 1, 0, 1), (1, 0, 1, 1), (0, 1, 1, 1)\}$, and~
$\Lambda'''_{3,5}  =  \{ (1, 2, 0, 0), (1, 0, 2, 0),\,\ldots,\, (0, 0, 2, 1) \}$.}
Referring to (\ref{coeff14}) and (\ref{hk}), we have for $k\ge 1$
\beq \label{coeff74}
h_k & = & \sum_{j=1}^n \left( \ell_j \sum_{\Lambda_{k,j}} \prod_{i=1, i\ne j}^n \mu_i^{\alpha_i}a_{\alpha_i}\right) \\
   & =  &
   \sum_{j=1}^n \ell_j \left( \sum_{\Lambda'_{k,j}} (\cdot ) +  \sum_{\Lambda''_{k,j}} (\cdot )+ 
         \sum_{\Lambda'''_{k,j}} (\cdot )\right)\nonumber \\
         & =: &
         \sum_{j=1}^n \ell_j \left(S_{1,j} + S_{2,j} + S_{3,j}\right), \quad \mbox{say}. \nonumber
\eeq
For the term $S_{2,j}$ in the middle, since $a_0=1$, we have $S_{2,j}=0$ when $k=1$ and for any $k \ge 2$
\nbeq
S_{2,j} & = & \sum_{\Lambda''_{k,j}} \prod_{i=1, i\ne j}^n \mu_i^{\alpha_i}a_{\alpha_i} \\
& = & a_0^{n-1-k}a_1^k 
\sum \  ^{\! \! \prime} (\mu_{i_1}\mu_{i_2}\cdots \mu_{i_k}) \nonumber \\
& = & 
a_1^k \mu^{-1}_j
\sum \  ^{\! \! \prime} (\mu_j\mu_{i_1}\mu_{i_2}\cdots \mu_{i_k}) \nonumber 
\neeq 
where the summation in $\sum \  ^{\! \! \prime}$  is over all $k$-tuples (with $i_j$th component dropped) $i_{1},\,\ldots,\,i_{j-1},i_{j+1}\ldots,\,i_{k}$,  such that $i_m\in \{1, 2,\,\ldots,\,n\}$ and $i_1<i_2<\ldots <i_k$.
Using that $\sum_{j=1}^n \ell_j \mu^{-1}_j=0$ by Lemma 2(ii) with $\lambda_i=\mu^{-1}_i$, we obtain for any $k \ge 2$
\be   \label{S2}
\sum_{j=1}^n \ell_j S_{2,j}  =  
a_1^k\left(\sum_{j=1}^n \ell_j\mu^{-1}_j \right) \sum \  ^{\! \! \prime \prime} (\mu_{i_1}\mu_{i_2}\cdots \mu_{i_k})
 =  0.
\ee
Here the summation in $ \sum \  ^{\! \! \prime\prime}$  is over all $k$-tuples $i_{1},i_{2},\,\ldots,\,i_k$, such that $i_m\in \{1, 2,\,\ldots,\,n\}$ and \mbox{$i_{1}< i_2<\ldots < i_k$}. 
For the first term $S_{1,j}$ in the last expression of (\ref{coeff74}), we have for any $k\ge 1$
\nbeq 
S_{1,j} & =  &
 \sum_{\Lambda'_{k,j}} \prod_{i=1, i\ne j}^n \mu_i^{\alpha_i}a_{\alpha_i}
 =   
a_0^{n-2} a_k  \sum_{i=1, i\ne j}^n \mu_i^k \nonumber \\
&  =   &
a_k \left(\sum_{i=1}^n \mu_i^k-\mu_j^k\right). \nonumber 
\neeq
Furthermore, since $\sum_{j=1}^n \ell_j=1$ by Lemma 2(i) with $\lambda_i=\mu^{-1}_i$, we have  for any $k\ge 1$
\beq \label{ck}
\sum_{j=1}^n \ell_j S_{1,j}
& = & a_k\sum_{j=1}^n \ell_j\left(\sum_{i=1}^n \mu_i^k-\mu_j^k\right)  \\
    & = & a_k\sum_{i=1}^n \mu^k_i\sum_{j=1}^n \ell_j-
    a_k\sum_{j=1}^n \ell_j\mu^k_j \nonumber \\
    & = & a_k\left(\sum_{i=1}^n \mu^k_i-\sum_{j=1}^n \ell_j\mu^k_j\right) \nonumber \\
    & =: & a_kc_k. \nonumber
\eeq

Lemma 2(iii) with $\lambda_i=\mu^{-1}_i$ implies that $c_1=0$ and $c_k< 0$ for any $k \ge 2$. It follows from (\ref{coeff74})--(\ref{ck})~that 
\be \label{coeff94}
h_k= c_ka_k
 +\sum_{j=1}^n \ell_j S_{3,j},
\ee
where $c_1=0$ and $c_k< 0$ for $k\ge 2$.

Let $k=1$. Since $h_1=0$ and the sets $\Lambda''_1$ and $\Lambda'''_2$ are empty, we obtain
$c_1 a_1=0$,
where $c_1=0$. 
Hence, there are no restrictions on the coefficient $a_1$, other than $a_1>0$, {since $X$ has positive mean}. Therefore, there is a number $\lambda^{-1}>0$ such that
\be \label{a1}
a_1=\lambda^{-1}>0.
\ee
Let $k=2$. Since the set $\Lambda'''_2$ is empty, Equation~(\ref{coeff94}) yields
$
h_2=c_2a_2 =0,
$
where recall that $c_2<0$. Thus,~ 
$a_2=0$. 
Next, applying (\ref{coeff94}) and taking into account that $h_k=0$ for $k\ge 2$,  we will show by induction that $a_k=0$ for any $k\ge 2$.
Assuming $a_k=0$ for $k=2,3,\,\ldots,\,r$, we will show that $a_{r+1}=0$. Indeed, by~(\ref{coeff94}) we have
\[
h_{r+1} =  c_{r+1}a_{r+1}+ \sum_{j=1}^n \left( \ell_j\sum_{\Lambda'''_{r+1,j}} \prod_{i=1, i\ne j}^n \mu_i^{\alpha_i}a_{\alpha_i}\right)
       = 
       c_{r+1}a_{r+1},
\]
because at least one index $\alpha_i$, satisfies $2\le \alpha_i\le r$ and hence $a_{\alpha_i}=0$, by~assumption. 
  Therefore,~\mbox{$h_{r+1}=c_{r+1}a_{r+1}=0$} and, since $c_{r+1}< 0$, we have $a_{r+1}=0$, which completes the induction.~Hence,
\be \label{ak}
a_k=0 \quad \mbox{for any} \quad k\ge 2.
\ee

The Equations~(\ref{a0}) and (\ref{a1})--(\ref{ak}) imply (\ref{claim}), which completes the proof of the theorem.
\end{proof}

\vspace{0.3cm}\begin{proof}[Proof of Theorem 2]
Taking into account (\ref{eqn_44}), similarly to (\ref{wphi}) and using integration-by-parts, we obtain
\nbeq 
\varphi(\mu_1 t)\varphi(\mu_2 t)\cdots \varphi(\mu_n t) & = & 
 \int_0^\infty {\rm e}^{-tx}g_n(x)\, dx 
     =  \int_0^\infty {\rm e}^{-tx}\sum_{j=1}^n \frac{\ell_j}{\mu_j}\overline{F}\left(\frac{x}{\mu_j}\right)\, dx  \nonumber \\
    & = & \sum_{j=1}^n \ell_j\int_0^\infty {\rm e}^{-tx}\frac{1}{\mu_j}\overline{F}\left(\frac{x}{\mu_j}\right)\, dx  \nonumber \\
     & = & \frac{1}{t}\sum_{j=1}^n \frac{\ell_j}{\mu_j} \left(1-\varphi(\mu_j t)\right). \nonumber
    \neeq
    
Using the fact that $\sum_{j=1}^n \ell_j/ {\mu_j}=0$ (see Lemma 2(ii)), this simplifies to
\be \label{wphi2}
\varphi(\mu_1 t)\varphi(\mu_2 t)\cdots \varphi(\mu_n t)
=- \frac{1}{t}\sum_{j=1}^n \frac{\ell_j}{\mu_j}\varphi(\mu_j t).
\ee

Dividing  both sides of (\ref{wphi2}) by $-\varphi(\mu_1 t)\varphi(\mu_2 t)\cdots \varphi(\mu_n t)/t$, for~$t>0$, we obtain
\be\label{eqn242}
-t = \sum_{j=1}^n \frac{\ell_j}{\mu_j}\prod_{i=1, i\ne j}^n \psi (\mu_i t), 
\ee
where, {as before}, $\psi =1/\varphi$. Consider the series
$\psi(t)=\sum_{k=0}^\infty a_kt^k$,
which is convergent by assumption.
To 
prove the theorem, it is sufficient to show that
$\psi(t)=1+t$, $t\ge 0$, or, equivalently, that~the
coefficients $\{a_k\}_{k=0}^\infty$ of the above series satisfy 
$a_0=1$, $a_1=1$, and~$a_k=0$ for $k\ge 2$. Clearly,~ 
\mbox{$a_0=1/\varphi(0)=1$}.
Recall that
\[
\Psi_j(t):= \prod_{i=1, i\ne j}^n  \psi (\mu_i t)\quad \mbox{and denote} \quad 
-Q(t)  :=  \sum_{j=1}^n \frac{\ell_j}{ \mu_j}\Psi_j(t) = - \sum_{k=0}^\infty q_kt^k.
\]
By (\ref{eqn242})  we have $Q(t)\equiv t$ and therefore $q_1=1$ and $q_k=0$ for all $k\ne 1$. We will express $q_k$ in terms of $a_j${'}s. 
Proceeding as in the proof of Theorem~1, applying Leibniz rule for differentiating {a} product of functions, and~using the same notation, we obtain for $k\ge 1$ that
\be \label{coeff742}
-q_k = \sum_{j=1}^n \frac{\ell_j}{ \mu_j}\left(S_{1,j} + S_{2,j} + S_{3,j}\right). \nonumber
\ee

As with (\ref{S2}), applying Lemma 2(ii), we obtain
\be   \label{S22}
\sum_{j=1}^n \frac{\ell_j}{ \mu_j} S_{2,j}  =  
a_1^k\left(\sum_{j=1}^n \frac{\ell_j}{\mu_j^2}  \right)\sum \  ^{\! \! \prime \prime} (\mu_{i_1}\mu_{i_2}\cdots \mu_{i_k})
 =  0, 
\ee
where the summation in $ \sum \  ^{\! \! \prime\prime}$  is over all $k$-tuples $i_{1},\,\ldots,\,i_k$, such that $i_m\in \{1,\,\ldots,\,n\}$ and $i_{1}<\ldots < i_k$. 
Furthermore, since $\sum_{j=1}^n \ell_j=1$ and $\sum_{j=1}^n \ell_j/ \mu^2_j=0$ by Lemma 2, we have  for any $k\ge 1$
\beq \label{ck2}
\sum_{j=1}^n \frac{\ell_j}{\mu_j}S_{1,j}
& = & a_k\sum_{j=1}^n \frac{\ell_j}{\mu_j}\left(\sum_{i=1}^n \mu_i^k-\mu_j^k\right)  \\
    & = & a_k\sum_{i=1}^n \mu_i^k\sum_{j=1}^n \frac{\ell_j}{\mu_j}-
    a_k\sum_{j=1}^n \ell_j\mu_j^{k-1} \nonumber \\
    & = & -a_k\sum_{j=1}^n \ell_j\mu_j^{k-1} \nonumber \\
    & =: & -a_kd_k. \nonumber
\eeq
It follows from (\ref{coeff742}) and (\ref{ck2}) that for $k\ge 1$,
\be \label{coeffd}
-q_k= -a_kd_k + \sum_{j=1}^n \frac{\ell_j}{\mu_j} S_{3,j}.
\ee

Let $k=1$.
Since $q_1=1$ and the set $\Lambda'''_1$ is empty, we obtain
$a_1d_1=1$, where $d_1=1$ by Lemma~2(iii). Therefore, $a_1=1$.
Let $k=2$. Since $\Lambda'''_2$ is empty, Equation~(\ref{coeffd}) yields
$
q_2=d_2a_2 =0,
$
where $d_2>0$ by Lemma~2(iii). Thus, $a_2=0$.
Assuming $a_k=0$ for $2\le k\le r$, we will show that $a_{r+1}=0$. Indeed, 
\[
q_{r+1} =   d_{r+1}a_{r+1}+ \sum_{j=1}^n  \frac{\ell_j}{\mu_j} S_{3,j}
       = 
       d_{r+1}a_{r+1},
\]
because at least one index $\alpha_i$, satisfies $2\le \alpha_i\le r$, in~which case $a_{\alpha_i}=0$, by~assumption. 
Therefore,~\mbox{$q_{r+1}=d_{r+1}a_{r+1}=0$} and, since $d_{r+1}< 0$, we have $a_{r+1}=0$, which completes the induction proof. Hence,
$
a_k=0$ for any $k\ge 2$.
Since $a_0=a_1=1$ and $a_k=0$ for $k\ge 2$, we obtain $\psi(t)=1+t$, which clearly completes the proof of the theorem.
 \end{proof}

\section{Concluding~Remarks}
\label{sec:4}

Arnold and Villase\~{n}or~\cite{AV13} proved 
that if $X_1$ and $X_2$ are two independent and non-negative random variables with common density $f$ and ${\mathsf E}[X_1]<\infty$, then
\[
X_1+\frac{1}{2}X_2 \qquad \mbox{has density} \qquad 2f(x)-2f(2x),\quad x>0,
\]
if and only if $X_1\sim {\rm Exp}(\lambda)$ for some $\lambda>0$. Motivated by this result, we extended it in two directions considering: (i) arbitrary number $n\ge 2$ of independent identically distributed non-negative random variables and (ii) linear combination of independent variables with arbitrary positive and distinct coefficients $\mu_1, \mu_2,\,\ldots,\,\mu_n$. Namely,
our main result is that 
\[
S_n=\mu_1X_1+\mu_2X_2+\ldots+\mu_nX_{n} \quad \mbox{has density}\quad g_n(x)= 
\sum_{j=1}^n \frac{\ell_j}{\mu_j}f\left(\frac{x}{\mu_j}\right)\quad x\ge 0,
\]
where 
$\ell_j =  \prod_{i=1, i\ne j}^n \mu_j(\mu_j-\mu_i)^{-1}$,
if and only if $X_i\sim {\rm Exp} (\lambda)$ for some $\lambda>0$.

In this paper, we dealt with the situation where the rate parameters $\lambda_i$ are all distinct from each other. The~other extreme case of equal $\lambda_i$'s is trivial. The~obtained characterization seems of interest on its own, but~it can also serve as a basis for further investigations of intermediate cases of mixed type with some ties and at least two distinct parameters (see~\cite{SKK16}).  Of~certain interest is also the case where  not all weights {$\mu_i$}'s are positive (see~\cite{LL19}).

\phantom
{It is known that the density of the sum of independent exponentials with distinct rate parameters is linear combination of the single variables densities.
Open problem (characterization of exponential) based on Exercise 13 in Feller.}

{\bf Funding} This research was funded in part by the National Scientific Foundation of Bulgaria at the Ministry of Education and Science, grant No KP-6-H22/3.

{\bf Acknowledgments} I thank Jordan Stoyanov for his mentorship, useful suggestions and critical comments on previous versions of the paper. The~author acknowledges the valuable suggestions from the anonymous~reviewers.




\end{document}